\documentclass[12pt]{article}
\usepackage{amsmath,amsfonts,amssymb,amsthm,amscd}
\title{Remarks on theories of free algebras and modules}
\date{\today}
\author{Anand Pillay\thanks{Supported by NSF grants DMS 1665035 and DMS-1760212}\\University of Notre Dame \and Philipp Rothmaler\\Bronx Community College, CUNY}

\newtheorem{Theorem}{Theorem}[section]
\newtheorem{Proposition}[Theorem]{Proposition}
\newtheorem{Definition}[Theorem]{Definition}
\newtheorem{Remark}[Theorem]{Remark}
\newtheorem{Lemma}[Theorem]{Lemma}

\newtheorem{Fact}[Theorem]{Fact}

\newtheorem{Question}[Theorem]{Question}
\newtheorem{Problem}[Theorem]{Problem}

\begin{document}
\maketitle

\begin{abstract} We ask some questions and make some observations about the (complete) theory $T_{\infty, {\mathcal V}}$ of free algebras in ${\mathcal V}$ on infinitely many generators, where ${\mathcal V}$ is a variety in the sense of universal algebra. We  focus on the case $T_{\infty, R}$ where ${\mathcal V}$ is the variety of $R$-modules ($R$ a ring).  Building on work in \cite{K-P}, we characterize when all models of $T_{\infty, R}$ are free, projective, flat, as well as when $T_{\infty, R}$ is categorical in a higher power.

\end{abstract}

\section{Introduction}

We discuss the theory of free algebras on infinitely many generators, with respect to a variety ${\mathcal V}$ in the sense of universal algebra, and with a specialization to the case where ${\mathcal V}$ is the variety of (left) $R$-modules for some unitary ring $R$. 

This work is, in a sense, a continuation of Baldwin-Shelah \cite{B-S}, Pillay-Sklinos \cite{P-S}, and Kucera-Pillay \cite{K-P}, but also connects with early work in the model theory of modules (such as \cite{S-E}) concerning when certain classes of $R$-modules (free, projective, flat) are elementary. This note answers some questions that  the first author raised in a seminar at Notre Dame when talking about the paper \cite{K-P}. 

The original context of Baldwin-Shelah was: suppose $L$ is a countable language with only function symbols, and $\mathcal V$ is a variety in this language ( in the sense of universal algebra).  Namely ${\mathcal V}$ is the class of $L$-structures defined by some set of equations  $\forall\bar x (t({\bar x}) = s({\bar x})$) where $s, t$ are terms of $L$.
Let $M$ be the free algebra in ${\mathcal V}$ on $\omega_{1}$ generators. Suppose that $M$ is $\omega_{1}$-saturated as an $L$-structure. What can we say, about $Th(M)$?  It was proved in \cite{B-S}  that the theory is $\omega$-stable, finite-dimensional (finitely many regular types, up to nonorthogonality), and with more information on the structure of the models. Somewhat more precise  statements and proofs were given in \cite{P-S}, where we also asked whether $Th(M)$ has to have finite Morley rank. 

There are various levels of generalizations of this, for example to uncountable languages as in \cite{K-P}, or to ``almost indiscernible theories" as in \cite{P-S} and \cite{K-P}.  Here, we will stick with free algebras, but try to generalize the problematic in different, and hopefully natural, directions, although we will work with possibly uncountable $L$.

Let us fix a variety ${\mathcal V}$ in a  language (or signature) $L$ of cardinality $\tau$. (Here the cardinality of $L$ denotes the number of function symbols in $L$ plus $\omega$.)   
For a cardinal $\kappa$ let $F_{\kappa}$ denote the free algebra in ${\mathcal V}$ on $\kappa$ many generators.
It is routine to show that for $\kappa$, $\lambda$ infinite $F_{\kappa}$ and $F_{\lambda}$ are elementarily equivalent.  So the common first order theory, which we will call $T_{\infty, {\mathcal V}}$, of the free algebras in ${\mathcal V}$ on infinitely many generators, is a complete theory in the language $L$, which is a canonical complete theory attached to the variety ${\mathcal V}$.

A somewhat imprecise question is
\begin{Problem} What can we say about the  theory $T_{\infty, {\mathcal V}}$?
\end{Problem}

As an example, when ${\mathcal V}$ is the variety of groups (in the language with identity, multiplication and inversion), then we know from Sela \cite{Sela},  that $T_{\infty, {\mathcal V}}$ is stable.  However this goes via proving the result for finitely generated free groups, and knowing the elementary inclusions of the free groups on $n$-generators for $n\geq 2$, whereby all the free groups on $\kappa$ generators for $\kappa \geq 2$ are stable. It is natural to ask whether there is a simpler approach just working with the free groups on infinitely many generators? For example the connectedness of the (theory of the) free group can be seen just by looking at the free groups on infinitely many generators \cite{Poizat}, \cite{Pillay-freegroup}. 

It is not hard to see that the condition that $F_{\tau^{+}}$ is ($\tau^{+}$)-saturated is equivalent to $F_{\kappa}$ being saturated for all $\kappa \geq \tau$  (see \cite{K-P}).  So we can interpret the condition that $F_{\tau^{+}}$ is saturated as saying that the ``standard" models of $T_{\infty, {\mathcal V}}$ are saturated.

When $L$ is countable,  the model theoretic consequences include that $T_{\infty, {\mathcal V}}$ is $\omega$-stable, with finitely many regular types and (roughly speaking) arbitrary models of $T_{\infty, \mathcal V}$ are generated by indiscernible sets \cite{B-S}, \cite{P-S}. 

For uncountable $L$, we only obtained (in \cite{K-P}) that $T_{\infty, {\mathcal V}}$ is superstable (but still ``finite-dimensional")  and it was asked in \cite{K-P} whether $T_{\infty, {\mathcal V}}$ is totally transcendental (every formula/type has ordinal valued Morley rank).  When  ${\mathcal V}$ is the variety of $R$-modules for a ring $R$, then$T_{\infty, \mathcal V}$ is indeed totally transcendental. 

In \cite{P-S} we asked whether $T_{\infty, \mathcal V}$ must have finite Morley rank (assuming countability of $L$), and a counterexample, was given in \cite{K-P}, for $\mathcal V$ the variety of $R$ modules for a certain ring $R$ of $2\times 2$ matrices. 

A natural question is to classify the varieties ${\mathcal V}$ with the property that the ``standard" models of $T_{\infty, {\mathcal V}}$ are saturated. As mentioned in \cite{B-S} these need not be all {\em stable varieties} (i.e. where all algebras are stable). 

A stronger condition (than $F_{\omega_{1}}$ being saturated) is that $T_{\infty, \mathcal V}$ is uncountably categorical.  Again, can we classify the varieties ${\mathcal V}$ with this property?

One may consider trying to say something about those ${\mathcal V}$ such that $T_{\infty, {\mathcal V}}$ is stable, and moreover to describe $T_{\infty, {\mathcal V}}$ in this situation, but as this includes the variety of all groups and the theory of the free group, it looks rather complicated.

On the other hand the condition that $T_{\infty, {\mathcal V}}$ is superstable would seem to be more tractable (and should imply totally transcendental). 

We refer to \cite{Pillay-book} for general stability theory and \cite{Prest} for the model theory of modules.

The first author would like to thank his earlier collaborators Sklinos and Kucera,  for their contribution to the overall project, as well as Prest for some earlier and helpful discussions.  Also thanks to John Baldwin  for pointing out possible connections to a question of Tarski.

\section{Modules}

A very specific class of varieties (in the sense of universal algebra), is the class of varieties of $R$-modules, as $R$ ranges over arbitrary, unitary rings (not necessarily commutative). 

It is interesting to study the questions mentioned in the introduction, for ${\mathcal V}$ the variety $_{R}Mod$ of (left) $R$ modules, some $R$.    This enterprise turns out to be closely related to classical results (\cite{S-E} for example) on describing rings $R$ for which certain classes of $R$-modules are elementary, as well as to the general model theory and  stability theory of modules. We will make some rather basic observations, which are thematically close to the content of Chapter 14 of \cite{Prest}  (see in particular the proof of Theorem 14.28 in \cite{Prest}).   The stability theory of flat modules was also studied by Rothmaler \cite{Rothmaler}.  In any case the general question of what can be said about the theory of free $R$-modules on infinitely many generators does not seem to have studied systematically before.

We will first describe briefly what was done in \cite{K-P}. We consider the category of left $R$-modules.  The language $_{R}L$  consists of $+,0,-$ and unary functions $\lambda_{r}$ for all $r\in R$ representing left multiplication by $r$.  $R$ itself is a (left) $R$-module and the free $R$-module on $\kappa$ many generators is just the direct sum $R^{(\kappa)}$ of $\kappa$ many copies of $R$.  We let $\tau$ denote the cardinality of $R$  (plus $\omega$). 

By a {\em projective} $R$ module,  we will mean a  direct summand of a free $R$-module. 

Consistent with notation in Section 1 will define $T_{\infty,R}$ to be the common (complete) theory of the free $R$-modules $R^{(\kappa)}$, $\kappa$ infinite.

\begin{Definition}
(i) $R$ is left perfect if $R$ has the descending chain condition on finitely generated left ideals,
\newline
(ii) $R$ is right coherent if every finitely generated right ideal is finitely presented.
\newline
(iii) $R$ is right artinian if $R$ satisfies the DCC on right ideals.

\end{Definition}

\begin{Remark} (i) As pointed out in \cite{Prest}, if $R$ is both left perfect and right coherent then it is totally trancendental as a a (left) module over itself.  In fact right coherence is equivalent to the $pp$-definable subgroups of $R$ being precisely the  finitely generated right ideals, and then left perfectness gives the dcc on pp-definable subgroups of $R$, which means totally transcendental. 
\newline
(ii) Right artinian imples left perfect and right coherent. 
\end{Remark}

It is well-known that any direct {\em product} $R^{I}$ of the (left) $R$ module $R$ is elementarily equivalent to the direct sum $R^{(I)}$.   Combining this with results of Chase  (as  in \cite{S-E}) has nice consequences:

Chase's theorem \cite{Chase} is the equivalence of (i) and (ii) below, and Sabbagh-Eklof \cite{S-E} add (iii)
\begin{Fact} The following are equivalent:
\newline
(i)  Any infinite direct product  $R^{I}$ is projective,
\newline
(ii) $R$ is left perfect and right coherent.
\newline
(iii) The class of projective $R$-modules is elementary. 
\end{Fact}

We first give a slight improvement of Theorem 3.15 from \cite{K-P}.
\begin{Proposition}  The following are equivalent:
\newline
(i)  $R^{(\tau^{+})}$  is saturated  (equivalently every $R^{(\kappa)}$ for $\kappa\geq \tau$  is saturated (in its own cardinality)),
\newline
(ii)  Every model of $T_{\infty, R}$ is projective,
\newline
(iii) The class of projective $R$-modules is elementary. 
\end{Proposition}
\begin{proof}
Assume (i). We know (\cite{K-P}) that $T_{\infty, R}$ is totally transcendental, in particular every model is pure-injective, also known as algebraically compact.   Let $M$ be an arbitrary model of $T_{\infty,R}$.  So $M$ is an elementary substructure of some saturated model, so of some $R^{(\kappa)}$, hence (by pure-injectivity) is a direct summand of $R^{(\lambda)}$ hence projective, so we get (ii).
\newline
Now assume (ii). As any infinite direct product $R^{I}$ is a model of $T_{\infty, R}$, we conclude that any such $R^{I}$ is projective, so we can apply Fact 2.3 to deduce (iii).
\newline
Finally assume (iii).  By Fact 2.3 and Remark 2.2, $T_{\infty, R}$ is totally transcendental.    We now want to prove that $R^{(\tau^{+})}$ is saturated.  As $T_{\infty,R}$ is totally transcendental, let $M$ a saturated model of cardinality $\tau^{+}$.  By assumption (iii) (in fact only (ii) is needed now), $M$ is projective, so a direct summand of some free $R$-module. An easy argument (basically downward Lowenheim Skolem) shows that $M$ is a direct summand of $R^{(\tau^{+})}$. As they both have the same theory, $M$ is an elementary substructure of $R^{(\tau^{+})}$. But for a totally transcendental theory of modules, any elementary extension of a $\tau^{+}$-saturated model is also $\tau^{+}$-saturated, and $R^{(\tau^{+})}$ is saturated, as required.

\end{proof} 

We also, by Fact 2.3, have the characterization of those $R$ for which $R^{(\tau^{+})}$ is saturated as precisely the left perfect and right coherent rings (which is what was stated in Theorem 3.15 of \cite{K-P}).  A special case of a question raised in \cite{P-S} is whether, under the conditions of Proposition 2.4, $T_{\infty, R}$ has finite Morley rank.  This would have been  a consequence of Exercise 2(b) in \cite{Prest}, but the Exercise is wrong, and  there is a counterexample (appearing in \cite{K-P} and also pointed out to us by Prest).

\vspace{5mm}
\noindent
We now want to slightly extend and/or generalize Proposition 2.4,  replacing projective by either  the weaker property flat, or  the stronger property free.

\begin{Definition} The $R$-module $M$ is flat iff whenever $m_{1},..,m_{n}\in M$, $r_{1},..,r_{n}\in R$ and $\sum_{i}r_{i}m_{i} = 0$, then there is a matrix $H$ over $R$ and sequence $m'$ from $M$ such that $Hm' = m$ and $rH = 0$.
\end{Definition}

There are other nice model-theoretic descriptions such as that for every $pp$ formula $\phi(x)$, $\phi(M) = \phi(R)M$.  (\cite{Rothmaler})

But in any case free implies projective implies flat (implies torsion-free).

Analogous to Fact 2.3 above, we have the following, where (i) iff (ii) is by Chase \cite{Chase} and (iii) is by Sabbagh and Eklof \cite{S-E}. 
\begin{Fact} The following are equivalent:
\newline
(i) Any infinite direct product $R^{I}$ is flat,
\newline
(ii) $R$ is right coherent,
\newline
(iii) The class of flat modules is elementary. 
\end{Fact}

We then obtain the following analogue of  Proposition 2.4:
\begin{Proposition} The following are equivalent:
\newline
(i)  All models of $T_{\infty, R}$ are flat,
\newline
(ii) The class of flat modules is elementary.

\end{Proposition}
\begin{proof} Assuming (i), any infinite direct product of copies of $R$ is a model of $T_{\infty,R}$, so flat, so apply Fact 2.6 to get (ii).  (ii) implies (i) is immediate.
\end{proof}

\begin{Question}  What  can we say about the theory $T_{\infty, R}$ when all models are flat?
\end{Question}

Finally we want to try to replace projective by free in Proposition 2.4.  

The situation is rather more complicated and we can distinguish various properties of $R$, as follows:

\vspace{2mm}
\noindent
(I) Any infinite direct product $R^{I}$ is free,
\newline
(II) $T_{\infty,R}$ is $\tau^{+}$-categorical (equivalently $\kappa$-categorical for all $\kappa \geq \tau^{+}$). 
\newline
(III) All models of  $T_{\infty,R}$ are free.
\newline
(IV)   The class of free $R$-modules is elementary. 

\vspace{2mm}
\noindent
The class of rings such that (IV) holds was described by Sabbagh-Eklof \cite{S-E}:
\begin{Fact} The following are equivalent:
\newline
(i) The class of free modules is elementary,
\newline
(ii) $R$ is right Artinian and either (a) $R$ is local, or (b) $R$ is finite and $R/J$ is simple where $J$ is the Jacobson radical of the the $R$-module $R$   (intersection of maximal submodules). 

\end{Fact} 

 Note that the condition (II) implies that the free modules $R^{(\kappa)}$ are saturated (for $\kappa\geq \tau^{+}$),  because (II) implies that every  model of cardinality $\geq \tau^{+}$ is saturated.   First:

\begin{Lemma}
Consider the properties (I), (II), (III), (IV) above. Then (IV) implies (III) implies (II) implies (I).
\end{Lemma}
\begin{proof} (IV) implies (III) is immediate.  Assuming (III)  there is clearly a unique model of $T$ of cardinality $\kappa$ for each $\kappa\geq \tau^{+}$ so we get (II).   And for II implies (I), it suffices to show that $R^{I}$ is free for large $I$. But then $R^{I}$ is a model of $T$ of cardinality $\geq \tau^{+}$ so has to be $R^{(\kappa)}$ where $\kappa = |R^{I}|$.
\end{proof} 

Now we discuss what can be said about the reverse implications and we obtain complete answers.  First we  characterize when $T_{\infty,R}$ is categorical in a higher power.

\begin{Proposition} $T_{\infty, R}$ is $\tau^{+}$-categorical if and only $R$ is left perfect and right coherent (equivalently every model of $T_{\infty, R}$ is projective), and there is a unique projective indecomposable $P$ (and $R = P^{(r)}$ for some integer $r > 1$).  
\end{Proposition}
\begin{proof} Assume $T_{\infty,R}$ to be $\tau^{+}$-categorical. Then   by Lemma 2.10 and Fact 2.3, $R$ is left perfect and right coherent.  As mentioned in \cite{S-E}, there are 
finitely many indecomposable projectives $P_{1},..,P_{k}$, and every model of $T_{\infty, R}$ is a direct sum of copies of the $P_{i}$ and moreover $R$ itself is of the form $P_{1}^{(r_{1})} 
\oplus ... \oplus  P_{k}^{(r_{k})}$ where $r_{1},..,r_{k}$ are positive integers. As $T$ is $\tau^{+}$-categorical   (so unidimensional) an easy argument as in \cite{S-E} shows that $k=1$.  
So we get the right hand side.
\newline
Conversely, assume the right hand side.   So every model of $T_{\infty, R}$ is projective, so a direct sum of copies of $P$ (noting $T_{\infty, R}$ is also tt).   So clearly $T_{\infty, R}$ is $\tau^{+}$-categorical. 
\end{proof} 

\begin{Remark} Assume the right hand side of Proposition 2.11.  Then there are two cases:
\newline
Case 1.  $R$ is finite.  Then the analysis of  \cite{S-E}  gives that $R/J$ is simple (as a ring) and of course $R$ is artinian.  And in this situation the class of free $R$ modules is elementary.
\newline
Case 2. $R$ is infinite, hence $P$ is infinite.  Let $P'$ be a large elementary extension of $P$, and we see as before that $P'$ is free, so a model of $T_{\infty, R}$. So $P$ is also a model of $T_{\infty, R}$. Now if $P$ is itself free, which means $r=1$, and $P = R$, then as in \cite{S-E}, $R$ is artinian and local, and again the class of free $R$-modules is elementary. 
If $P$ is not free then we see that  not every model of $T_{\infty, R}$ is free and so also the class of frees is not elementary.
\end{Remark}

So from Remark 2.12 we conclude: 
\begin{Proposition}
(i)  Let $R$ be left coherent, and right perfect,  with a unique projective indecomposable module $P$ which is infinite but not equal to $R$.  Then $T_{\infty, R}$ is $\tau^{+}$-categorical, but not every model of $T_{\infty, R}$ is free.
\newline
(ii) Suppose that every model of $T_{\infty, R}$ is free, then the class of free $R$-modules is elementary.
\end{Proposition}

Let $F$ be an infinite field, and $R$ the ring of $2\times 2$ matrices over $F$. Then $R$ satisfies the conditions in Proposition 2.13(i).  The unique projective indecomposable $P$  is the set of $2\times 1$ matrices over $F$  (as a left $R$ module). So $P$ is not free (but note  $P\oplus P= R$).

Finally, by the discussion on p. 640 of \cite{S-E} we see that there is  $R$ (even commutative) such that  $R^{I}$ is free for all infinite $I$, but $T_{\infty, R}$ is not $\tau^{+}$-categorical. 

Hence in terms of trying to reverse the arrows in Lemma 2.10  we have
\begin{Proposition} (i)  (III) and (IV) are equivalent.
\newline
(ii)  The implications (III) implies (II), and (II) implies (I) are strict. 
\end{Proposition} 

\section{Further remarks}
Fact 2.6 states that the class of flat $R$ modules is elementary if and only if $R$ is right coherent. Bass \cite{Bass} proved that flat= projective if and only if $R$ is left perfect. So by Fact 2.3 we have the conclusion that the class of projectives is elementary iff the class of flats is elementary and flat = projective (right implies left being trivial).

Is there a similar relationship between projectives and frees.  In fact the frees being elementary does not imply that projective = free.  The case of Fact 2.9 (ii)(b) where $R$ is finite, $R/J$ is simple (and $R$ is not local) will give a counterexample. 

However we do have:

\begin{Lemma} Suppose $R$ is infinite. Then the class of free $R$ modules is elementary if and only if the class of projective $R$-modules is elementary and projective $R$-modules are free.
\end{Lemma}
\begin{proof}  The nontrivial direction is left to right.  Assume that the class of free $R$ modules is elementary.  By 2.10 and 2.4 the class of projective $R$-modules is elementary. 

\end{proof}

 Let us mention in passing that describing the theory $T_{\infty, R}$ when $R$ is left perfect may be intereresting. For example in Eklof-Sabbagh it is observed that under this hypothesis any projective $R$-module is uniquely a direct sum of indecomposable projectives, and moreover (in fact as a consequence) any indecomposable projective is a direct summand of $R$ (as left $R$-modules). 

\vspace{5mm}
\noindent
When the first author gave an online talk on the material in this paper, John Baldwin asked about the connection with a conjecture of Tarski that if ${\mathcal V}$ is a variety (in a countable language) which is ``uncountably categorical" (i.e. for some/any uncountable $\kappa$ any two algebras in $\mathcal V$ of cardinality $\kappa$ are isomorphic) then any infinite algebra in ${\mathcal V}$ is free. The only place we saw this conjecture mentioned was in the paper \cite{Baldwin-Lachlan} where the authors Baldwin and Lachlan stated that they had heard about the conjecture from McNulty.  Actually a counterexample to the conjecture appears in Section 1 of \cite{Baldwin-Lachlan} where a variety is obtained from the variety of vector spaces over an infinite field, by a certain construction.

In any case the example mentioned above where $R$ is the ring of $2\times 2$ matrices over an infinite (say countable) field $F$, gives another (fairly natural)  counterexample to this conjecture of Tarski.  The reader is referred to Example 2 on page 9 of \cite{Prest}, from which one sees that with $R$ above, {\em any} $R$-module is (uniquely) a direct sum of copies of the unique indecomposable $P$ which is the set of $2\times 1$ matrices over $F$. So ${\mathcal V}$ (not only $T_{\infty, R}$) is uncountably categorical, but $P$ is an infinite algebra in ${\mathcal V}$ which is not free.

\end{document}